%% file: hopfmon.tex
\documentclass{FPSAC2020}

\usepackage{tikz}
\usepackage{verbatim}
\usetikzlibrary{arrows, shapes}

\newtheorem{theorem}{Theorem}
\newtheorem{proposition}[theorem]{Proposition}
\newtheorem{lemma}[theorem]{Lemma}

\newtheorem*{theorem*}{Theorem}

\theoremstyle{definition}

\newtheorem{definition}[theorem]{Definition}

\newcommand{\spe}[1]{\mathbf{#1}}

\DeclareMathOperator{\Int}{Int}
\DeclareMathOperator{\lk}{link}

\DeclareMathOperator{\dcone}{dcone}
\usepackage{lipsum}

\title[Cohen-Macaulay Hopf Monoids]{On Cohen-Macaulay Hopf monoids in species}

\author[J. A. White]{Jacob A. White\thanks{\href{mailto:jacob.white@utrgv.edu}{jacob.white@utrgv.edu}}\addressmark{1}}

\address{\addressmark{1}School of Mathematical and Statistical Sciences, University of Texas Rio Grande Valley, Edinburg, TX, USA}

\received{\today}

\abstract{We study Cohen-Macaulay Hopf monoids in the category of species. The goal is to apply techniques from topological combinatorics to the study of polynomial invariants arising from combinatorial Hopf algebras. Given a polynomial invariant arising from a linearized Hopf monoid, we show that under certain conditions it is the Hilbert polynomial of a relative simplicial complex. If the Hopf monoid is Cohen-Macaulay, we give necessary and sufficient conditions for the corresponding relative simplicial complex to be relatively Cohen-Macaulay, which implies that the polynomial has a nonnegative $h$-vector. We apply our results to the weak and strong chromatic polynomials of acyclic mixed graphs, and the order polynomial of a double poset.}


\begin{document}

\maketitle

\section{Introduction}
Suppose we have a sequence $A_0, A_1, \ldots$ of finite sets and, for each $a \in A_i,$ we have a polynomial $p(a,x)$ with the property that $p(a,k) \in \mathbb{N}$ for all $k \in \mathbb{N}$. Recall that the $W$-transform of a polynomial $p(a,x)$ is given by 
\[W(a, x) = (1-x)^{d+1} \sum_{k \geq 0} p(a, k) x^k \] 
where $d$ is the degree of $p(a_i, k)$. The $W$-transform is a polynomial of degree $d$ with integer coefficients. Write $W(a, x) = \sum_{k=0}^d h_k x^k$, and refer to $(h_0, \ldots, h_d)$ as the $h$-vector of $p(a, x)$. When is the $h$-vector nonnegative? When this happens, we say that $p(a, k)$ is $h$-positive. There is a similar concept for quasisymmetric functions called $F$-positivity, where $F$ refers to the basis of fundamental quasisymmetric functions.

A classical example of the above problem is where $A_n$ is the set of graphs with vertex set $[n] = \{1, \ldots, n \}$, and $p(G, k)$ is the chromatic polynomial. In this case, the $h$-vector is always nonnegative, as was shown by Brenti \cite{brenti}.

We will give sufficient conditions for nonnegativity of the $h$-vector for the weak and strong chromatic polynomials of an acyclic mixed graph, and for the order polynomial of a double poset. Our conditions also imply $F$-positivity for corresponding quasisymmetric functions.

Our first observation is that all of these polynomial invariants arise from the theory of Hopf monoids. If we have collections of labeled objects, with rules for how to combine the combinatorial objects and decompose them, and the rules are `well-behaved', then we have a linearized Hopf monoid. We review the definition in Section \ref{sec:combhopfmon}. If we have some distinguished Hopf submonoid $\spe{S}$ satisfying certain conditions, then there is a polynomial invariant $\chi^{\spe{S}}_{\spe{H}}(\spe{h}, k)$ which counts decompositions of an $\spe{H}$-structure into $\spe{S}$-structures. For example, graphs form a linearized Hopf monoid $\spe{G}$, and if $\spe{S}$ is the submonoid of edgeless graphs, then $\chi^{\spe{S}}_{\spe{G}}(\spe{g},k)$ counts decompositions into edgeless graphs, and thus is the chromatic polynomial. We refer the polynomial $\chi^{\spe{S}}_{\spe{H}}(\spe{h}, k)$ as the characteristic polynomial of $\spe{h}$ with respect to $\spe{S}$.

We recently \cite{white-2} gave conditions for when characteristic polynomials are Hilbert polynomials. Given a linearized Hopf monoid $\spe{H}$, a \emph{geometric Hopf submonoid} $\spe{S}$, and an $\spe{H}$-structure $\spe{h}$, we constructed a relative simplicial complex $(\Sigma(\spe{h}), \Gamma_{\spe{S}}(\spe{h}))$ such that $\chi^{\spe{S}}_{\spe{H}}(\spe{h},k+1)$ is the Hilbert polynomial of the double cone of $(\Sigma(\spe{h}), \Gamma_{\spe{S}}(\spe{h}))$. The complex $\Gamma_{\spe{S}}(\spe{h})$ is a generalization of the coloring complex of a graph \cite{steingrimsson}.

Thus, if our sets $A_i$ form a linearized Hopf monoid and our polynomial invariant is a characteristic polynomial, then we are studying Hilbert polynomials of relative simplicial complexes. If a relative simplicial complex is relatively Cohen-Macaulay, then the Hilbert polynomial is $h$-positive. In our case, the corresponding quasisymmetric functions are also $F$-positive.
We arrive at the problem we study in this paper:
find nice combinatorial conditions on $\spe{H}$, $\spe{S}$ and $\spe{h}$ to ensure that $(\Sigma(\spe{h}), \Gamma_{\spe{S}}(\spe{h}))$ is relatively Cohen-Macaulay, to obtain a new tool for showing $h$-positivity. Our main assumption is that $\spe{H}$ is Cohen-Macaulay, which means that $\Sigma(\spe{h})$ is Cohen-Macaulay for every $\spe{h}$. This tends to be true for most examples we know of.

\begin{theorem}
Let $\spe{H}$ be a linearized Hopf monoid and let $\spe{S}$ be a geometric Hopf submonoid. Let $N$ be a finite set, and let $\spe{h} \in \spe{H}_N$. Then $\spe{h}$ is relatively Cohen-Macaulay with respect to $\spe{S}$ if and only if the following condition is satisfied:
\begin{itemize}
    \item For all $S \subset T \subseteq N$, if $|T \setminus S| \geq 2$, then $\Gamma_{\spe{S}}(\spe{h}|_T /S)$ is a connected pure simplicial complex of dimension $\dim \Sigma(\spe{h}|_T/S) - 1$.
\end{itemize}
\label{thm:main}
\end{theorem}
Instead of trying to determine which polynomials are $h$-positive, as was done with the chromatic polynomial \cite{brenti}, or trying to determine which subcomplexes $\Gamma_{\spe{S}}(\spe{h})$ are shellable from first principles as was done for the coloring complex in \cite{hultman}, we merely have to determine when $\Gamma_{\spe{S}}(\spe{h})$ is pure, connected, and has the right dimension. 

We apply our techniques to three new examples: the weak and strong chromatic polynomials of acyclic mixed graphs, and the order polynomial of a double poset. In all three examples, we obtain necessary and sufficient \emph{combinatorial} conditions for determining whether or not the relative simplicial complexes are Cohen-Macaulay. Our approach also implies known results for coloring complexes of hypergraphs \cite{hypergraph}, and can be used to prove $h$-positivity results there. 

The paper is organized as follows. In Section \ref{sec:combhopfmon}, we review linearized Hopf monoids, along with the special example of acyclic mixed graphs. We also define the characteristic polynomials and quasisymmetric functions, along with the strong chromatic polynomial. In Section \ref{sec:simplicial}, we define our simplicial complexes, and the definition of Cohen-Macaulay Hopf monoid. In that section we also focus on the example of acyclic mixed graphs, and show that strong chromatic polynomial of an acyclic mixed graph is $h$-positive. Then we apply our theorem to the weak chromatic polynomial and to double posets. In Section \ref{sec:proof}, we sketch a proof of our theorem.

\section{linearized Hopf monoids}
\label{sec:combhopfmon}

In this section, we study Hopf monoids in the category of linear species. The motivation is that many examples of combinatorial Hopf algebras come from linearized Hopf monoids in set species. Moreover, the coproduct of a basis element is a multiplicity-free sum of tensors of basis elements. 
A \emph{set species} is a functor $\spe{F}: Set \to Set$ from the category of finite sets with bijections, to the category of finite sets with bijections. 
A \emph{linear species} is a functor $\spe{F}: Set \to Vec$ from the category of finite sets with bijections to the category of finite dimensional vector spaces over a field $\mathbb{K}$ and linear transformations.
Given a set species $\spe{F}$, there is an associated linear species $\mathbb{K}\spe{F}$ called the \emph{linearization}: we define $(\mathbb{K}\spe{F})_N$ to be the vector space with basis $\spe{F}_N$. We refer to $\spe{f}$ as an $\spe{F}$-structure if there exists a finite set $N$ such that $\spe{f} \in \spe{F}_N$.

A \emph{Hopf monoid} is a Hopf monoid object in the category of linear species \cite{aguiar-mahajan-1}. We give some of the structural definition and axioms related to associativity and compatability. There are also unit, counit morphisms and antipode axioms. We refer to \cite{aguiar-mahajan-2,aguiar-ardila} for more details.
For every pair of disjoint finite sets $M, N$, we have linear transformations $\mu_{M,N}: \spe{H}_M \otimes \spe{H}_N \to \spe{H}_{M \sqcup N}$ and $\Delta_{M,N}: \spe{H}_{M \sqcup N} \to \spe{H}_M \otimes \spe{H}_N$. We refer to $\mu$ as \emph{multiplication} and $\Delta$ as \emph{comultiplication}.
We require several axioms, including:
\begin{enumerate}
    \item $\mu_{L, M \sqcup N} \circ 1_L \otimes \mu_{M, N} = \mu_{L \sqcup M, N} \circ \mu_{L,M} \otimes 1_N$
    \item $1_L \otimes \Delta_{M, N} \circ \Delta_{L, M \sqcup N} = \Delta_{L,M} \otimes 1_N \circ \Delta_{L \sqcup M, N}$.
    \item $\Delta_{A \sqcup C, B \sqcup D} \circ \mu_{A \sqcup B, C \sqcup D} = \mu_{A,C} \otimes \mu_{B,D} \circ 1_A \otimes \tau_{B,C} \otimes 1_D \circ \Delta_{A,B} \otimes \Delta_{C,D}$.
\end{enumerate}
Note that these are equalities of functions, and $1$ is the identity map.
We let $\Delta_{L, M, N} = 1_L \otimes \Delta_{M,N} \circ \Delta_{L,M \sqcup N}$, and $\spe{x} \cdot \spe{y} = \mu_{M,N}(\spe{x} \otimes \spe{y}) $.

Let $\spe{H}$ be a species, and suppose that the linearization $\mathbb{K}\spe{H}$ is a Hopf monoid.
We say that $\spe{H}$ is a (weakly) \emph{linearized} Hopf monoid if $\mathbb{K}\spe{H}$ is a Hopf monoid, and:
\begin{enumerate}
    \item for every pair of disjoint finite sets $N, M$, and every $\spe{x} \in \spe{H}_M, \spe{y} \in \spe{H}_N$, we have $\spe{x} \cdot \spe{y} \in \spe{H}_{M \sqcup N}$.
    \item for every pair of disjoint finite sets $M, N$ and every $\spe{x} \in \spe{H}_{M \sqcup N}$, if $\Delta_{M,N}(\spe{x}) \neq 0$ then there exists $\spe{x}|_M \in \spe{H}_M$ and $\spe{x}/M \in \spe{H}_N$ with $\Delta_{M,N}(\spe{x}) = \spe{x}|_M \otimes \spe{x}/M$.
\end{enumerate}
In other words, the monoid structure is also linearized, and the coproduct $\Delta_{M,N}$ sends basis elements to simple tensors of basis elements or $0$. The notion of linearized Hopf monoids have been introduced before \cite{aguiar-mahajan-2}. Our definition is weaker as we allow $\Delta_{M,N}(\spe{h}) = 0$. We refer to $\spe{h}|_T/S$ as a \emph{minor} of $\spe{h}$ for any $S \subseteq T \subseteq N$. 
Most Hopf monoids that have been studied are Hadamard products of linearized Hopf monoids and their duals. Examples of linearized Hopf monoids include the Hopf monoid of graphs, posets, matroids, hypergraphs, set partitions, linear orders, and generalized permutohedra. In fact, almost every Hopf monoid studied in \cite{aguiar-ardila} is a linearized Hopf monoid. We will often describe the multiplication and comuliplication operations on $\spe{H}$, and leave it to the reader to see that the induced maps on $\mathbb{K}\spe{H}$ turn $\spe{H}$ into a linearized Hopf monoid.

Our primary example of a linearized Hopf monoid will be the linearized Hopf monoid of acyclic mixed graphs.
An \emph{acyclic mixed graph} on $N$ has both directed edges and undirected edges, without any directed cycles. We require that our acyclic mixed graphs are simple: there is at most one edge between any two vertices.

Now we describe a Hopf monoid structure on the species of acyclic mixed graphs.
Let $\spe{M}_N$ be the set of acyclic mixed graphs on $N$. Given $\spe{g} \in \spe{M}_M$ and $\spe{h} \in \spe{M}_N$, where $M$ and $N$ are disjoint sets, we let $\spe{g} \cdot \spe{h}$ be the disjoint union. This defines our multiplication operation.
Now we define the comultiplication operation. Let $\spe{g} \in \spe{M}_{M \sqcup N}$. If there exists a directed edge of the form $(m,n)$ where $m \in M$ and $n \in N$, then $\Delta_{M,N}(\spe{g}) = 0$. Otherwise, we let $\Delta_{M,N}(\spe{g}) = \spe{g}|_M \otimes \spe{g}|_N$, where $\spe{g}|_M$ is the induced subgraph on $M$. For example, given the acyclic mixed graph $\spe{g}$, on the left in Figure \ref{fig:graph}, we see that $\Delta_{\{a,b,d \}, \{c\}}(G) = \spe{g}_1 \otimes \spe{g}_2$, where $\spe{g}_1$ is the acyclic mixed graph in the middle of Figure \ref{fig:graph}, and $\spe{g}_2$ is the vertex $c$ by itself. On the other hand, $\Delta_{\{d\}, \{a,b,c\}}(\spe{g}) = 0$.
With our multiplication and comultiplication operations, $\spe{M}$ is a linearized Hopf monoid.
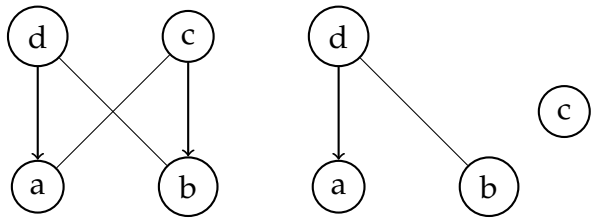
\begin{figure}
\begin{center}
\begin{tikzpicture}
\node[thick, draw=black,  fill = white, circle] (a) at (-2,-1) {a};
\node[thick, draw=black,  fill = white, circle] (b) at (0,-1) {b};
\node[thick, draw=black,  fill = white, circle] (c) at (0,1) {c};
\node[thick,  draw=black, fill = white, circle] (d) at (-2,1) {d};

\draw[-, line cap = round] (a) -- (c);
\draw[-, line cap = round] (b) -- (d);
\draw[->, thick, line cap = round] (c) -- (b);
\draw[->, thick, line cap = round] (d) -- (a);

\node[thick, draw=black,  fill = white, circle] (a2) at (2,-1) {a};
\node[thick, draw=black,  fill = white, circle] (b2) at (4,-1) {b};
\node[thick, draw=black,  fill = white, circle] (c2) at (5,0) {c};
\node[thick,  draw=black, fill = white, circle] (d2) at (2,1) {d};

\draw[-, line cap = round] (b2) -- (d2);

\draw[->, thick, line cap = round] (d2) -- (a2);
\end{tikzpicture}
\end{center}
\caption{an acyclic mixed graph, and some of its induced subgraphs.}
\label{fig:graph}
\end{figure}

A \emph{geometric Hopf submonoid} $\spe{S}$ of $\spe{H}$ is a species such  that:
\begin{enumerate}
    \item $\spe{S}_N \subseteq \spe{H}_N$ for all $N$
    \item For $\spe{x} \in \spe{H}_M$ and $\spe{y} \in \spe{H}_N$, we have $\spe{x}\cdot \spe{y} \in \spe{S}_{M \sqcup N}$ if and only if $\spe{x}  \in \spe{S}_{M}$ and $\spe{y} \in \spe{S}_{N}$.
    \item For $\spe{x} \in \spe{S}_{M \sqcup N}$, if $\Delta_{M,N}(\spe{x}) \neq 0$, then $\spe{x}|_M \in \spe{S}_M$ and $\spe{x}/M \in \spe{S}_N$.
    \end{enumerate}
The motivation for the term \emph{geometric} will be revealed in the next section.
For the Hopf monoid of acyclic mixed graphs, we can let $\spe{S}_N = \{D_N \}$ where $D_N$ is the edgeless graph. Then $\spe{S}$ is a geometric Hopf submonoid.

For this paper, we focus on quasisymmetric function and polynomial invariants associated to a \emph{geometric Hopf submonoid}. The fact that our constructions form quasisymmetric functions relies on the existence of certain characters, the Fock functors of Aguiar and Mahajan \cite{aguiar-mahajan-1}, and the theory of combinatorial Hopf algebras \cite{aguiar-bergeron-sottile}. However, we will limit ourselves to the definitions of the invariants for this extended abstract. 

Let $N$ be a finite set, and let $\spe{h} \in \spe{H}_N$. Let $f: N \to \mathbb{N}$. We say that $f$ is $\spe{S}$-proper if $\spe{h}|_{f^{-1}([i])} / f^{-1}([i-1])$ exists and is an $\spe{S}$-structure for all $i$.
For example, consider an acyclic mixed graph $\spe{g}$, and let $\spe{S}$ be the geometric Hopf submonoid of edgeless graphs.
In order for a function $f$ to be $\spe{S}$-proper we see that, for every $i$, $f^{-1}([i])$ must not contain a vertex $u$ that is part of a directed edge $(u,v)$ with $f(v) > i$. If $f(v) = j > i$, then $\spe{g}|_{f^{-1}([j])} / f^{-1}([j-1])$ does not exist. Moreover, the induced subgraph on $f^{-1}(i)$ has no edges. This implies that, for every directed edge $(u,v)$, we must have $f(u) < f(v)$. Also, for every undirected edge $uv$, we must have $f(u) \neq f(v)$. Hence $f$ is a \emph{strong coloring} of $\spe{g}$ as defined in \cite{beck-et-al}. 

Let $ \{x_i: i \in \mathbb{N} \}$ be a set of commuting indeterminates. 
\begin{definition}
Let $\spe{H}$ be a linearized Hopf monoid and $\spe{S}$ be a Hopf submonoid. Fix a finite set $N$, and $\spe{h} \in \spe{H}_N$. 
Then  the \emph{characteristic quasisymmetric function} with respect to $\spe{S}$ is given by \[\Psi_{\spe{H}}^{\spe{S}}(\spe{h}) = \sum_{f}  \prod_{n \in N} x_{f(n)} \]
where the sum is over $\spe{S}$-proper functions.
For $k \in \mathbb{N}$, we define $\chi_{\spe{H}}^{\spe{S}}(\spe{h}, k)$ to be the number of $\spe{S}$-proper functions with codomain $[k]$. This is the \emph{characteristic polynomial} (in the variable $k$) with respect to $\spe{S}$.
\end{definition}

For example, $\chi^{\spe{S}}_{\spe{M}}(\spe{g},k)$ counts the number of strong colorings, and is called the \emph{strong chromatic polynomial}, and $\Psi^{\spe{S}}_{\spe{M}}(\spe{g})$ is the strong chromatic quasisymmetric function. When $\spe{g}$ has no directed edges, we recover the usual chromatic polynomial, and Stanley's chromatic symmetric function \cite{stanley-coloring-1}. We will use $\bar{\chi}(\spe{g},k)$ to denote the strong chromatic polynomial. The strong chromatic polynomial of the graph $\spe{g}$ on the left in Figure \ref{fig:graph} has $h$-vector $(0,1,2,3)$.

For the reader who is more familiar with combinatorial Hopf algebras, we can give more explanation about the connection between $\Psi_{\spe{H}}^{\spe{S}}$ and the theory of combinatorial Hopf algebras. Namely, given the linearized Hopf monoid $\spe{H}$, we know that the symmetric group $S_n$ acts on $\spe{H}_{[n]}$. We let $H_n$ be the vector space whose basis elements are the orbits of this action. let $H = \bigoplus_{n \geq 0} H_n$. Then it is known that $H$ is a graded, connected Hopf algebra. Given an $\spe{H}$-structure $\spe{h}$, we define $\varphi(\spe{h}) = 1$ if $\spe{h}$ is an $\spe{S}$-structure, and $0$ otherwise. Then $\varphi: H \to \mathbb{K}$ is a character. It follows that there is a unique Hopf algebra homomorphism $\Psi_{\varphi}:H \to QSym$. Then we have $\Psi_{\spe{H}}^{\spe{S}}(\spe{h}) = \Psi_{\varphi}(\spe{h})$. In our case, it turns out this more algebraic definition can be rephrased purely as a generating function counting certain types of functions \cite{white-2}.

\section{Cohen-Macaulay Hopf Monoids}
\label{sec:simplicial}
We define Cohen-Macaulay complexes and Cohen-Macaualay Hopf monoids. The former is a well-known concept, while the latter is new. We assume the reader is familiar with terminology regarding simplicial complexes such as the \emph{link} of a face $\sigma$, which we denote $\lk_{\Sigma}(\sigma)$, reduced homology $\widetilde{H}_i(\Sigma)$, and relative homology.
A simplicial complex $\Sigma$ is \emph{Cohen-Macaulay} if $\dim \widetilde{H}_i(\lk_{\Sigma}(\sigma)) = 0$ for $i < \dim \lk_{\Sigma}(\sigma)$ for every $\sigma \in \Sigma$.

Let $\spe{H}$ be a linearized Hopf monoid. Let $N$ be a finite set, and consider $\spe{h} \in \spe{H}_N$. We define the simplicial complex $\Sigma(\spe{h})$. The vertex set $V(\spe{h})$ consists of all nonempty $S \subset N$ for which $\Delta_{S, N \setminus S}(\spe{h}) \neq 0$. We partially order $V(\spe{h})$ by inclusion, and let $\Sigma(\spe{h})$ be the order complex, which consists of chains $ \emptyset \subset S_1 \subseteq S_2 \subseteq \cdots \subseteq S_k \subset N$ such that $\Delta_{S_i, N \setminus S_i}(h) \neq 0$ for all $i$. 

As an example, let $\spe{M}$ be the Hopf monoid of acyclic mixed graphs. Let $\spe{g}$ be the acyclic mixed graph on the left in Figure \ref{fig:graph}. Then $\Sigma(\spe{g})$ is depicted in Figure \ref{fig:graphcomplex}.

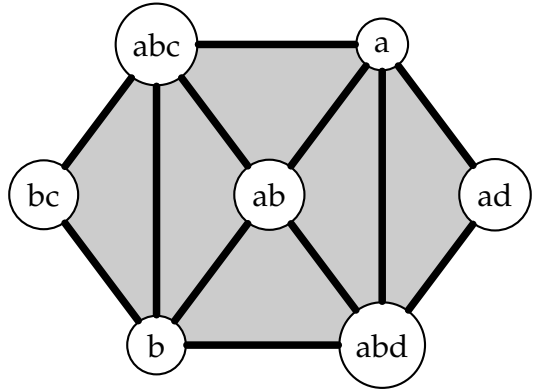
\begin{figure}
\begin{center}
\begin{tikzpicture}[z={(-.5,-.28)}]

\draw[fill=black, opacity=.2, rounded corners] (-3,0) -- (-1.5,-2) -- (1.5,-2) -- (3,0) -- (1.5, 2) -- (-1.5,2) -- cycle;

\node[thick,  draw=black, fill = white, circle] (bc) at (-3,0) {bc};
\node[thick,  draw=black, fill = white, circle] (b) at (-1.5,-2) {b};
\node[thick,  draw=black, fill = white, circle] (abd) at (1.5,-2) {abd};
\node[thick,  draw=black, fill = white, circle] (ad) at (3,0) {ad};
\node[thick,  draw=black, fill = white, circle] (a) at (1.5,2) {a};
\node[thick,  draw=black, fill = white, circle] (abc) at (-1.5,2) {abc};
\node[thick,  draw=black, fill = white, circle] (ab) at (0,0) {ab};
\draw[-, line width = 1mm, line cap = round] (bc) -- (abc) -- (a) -- (ad) -- (abd) -- (b) -- (bc);
\draw[-, line width = 1mm, line cap = round] (b) -- (abc) -- (ab) -- (abd) -- (a) -- (ab) -- (b);

\end{tikzpicture}
\end{center}
\caption{An example of a simplicial complex coming from an $\spe{M}$-structure.}
\label{fig:graphcomplex}
\end{figure}
The relevance of $\Sigma(\spe{h})$ is that the algebraic structure of $\Sigma(\spe{h})$ encodes combinatorial information about $\spe{h}$. Namely, $\chi^{\spe{H}}_{\spe{H}}(\spe{h},k+1)$ is the Hilbert polynomial of the double cone over $\Sigma(\spe{h})$ \cite{white-2}.

We say that $\spe{H}$ is \emph{Cohen-Macaulay} if $\Sigma(\spe{h})$ is Cohen-Macaulay for every $\spe{H}$-structure $\spe{h}$.
\begin{theorem}
Let $\spe{H}$ be a linearized Hopf monoid. Then $\spe{H}$ is Cohen-Macaulay if and only if for every finite set $N$, every $h \in \spe{H}_N$, and every $i < |N|$, we have $\dim \widetilde{H}_i(\Sigma(\spe{h})) = 0$.
\end{theorem}

Note that the advantage of the theorem is that we do not have to consider homology groups of links. For many examples of linearized Hopf monoids appearing in the literature, all the complexes $\Sigma(\spe{h})$ are contractible, and our theorem immediately implies that $\spe{H}$ is thus Cohen-Macaulay.

\begin{proposition}
The species $\spe{M}$ of acyclic mixed graphs is a Cohen-Macaulay Hopf monoid.
\end{proposition}
\begin{proof}
 For an acyclic mixed graph $\spe{g}$, we define a partial order $\spe{p}(\spe{g})$ such that $\Sigma(\spe{g}) = \Sigma(\spe{p}(\spe{g}))$. Given nodes $m,n \in N$, we say $m \leq n$ in $\spe{p}$ if there exists a directed path (using only directed edges) from $n$ to $m$ in $\spe{g}$. This defines a partial order $\spe{p}(\spe{g})$, with $\Sigma(\spe{p}(\spe{g})) = \Sigma(\spe{g})$. It is known that $\Sigma(\spe{p}(\spe{g}))$ is a triangulation of the poset polyhedron \cite{aguiar-ardila}. Hence $\Sigma(\spe{p}(\spe{g}))$ is contractible for any partial order, and we see that $\Sigma(\spe{g})$ has trivial homology below top dimension for every acyclic mixed graph $\spe{g}$. In particular, $\spe{M}$ is Cohen-Macaulay.
\end{proof}

Now we discuss relative simplicial complexes.
A relative simplicial complex on a set $V$ consists of a pair $(\Sigma, \Gamma)$ of simplicial complexes such that every face of $\Gamma$ is a face of $\Delta$. A relative complex $(\Gamma, \Sigma)$ is \emph{relatively Cohen-Macaulay} if $\dim \widetilde{H}_i(\lk_{\Sigma}(\sigma), \lk_{\Gamma}(\sigma)) = 0$ for $i < \dim \lk_{\Sigma}(\sigma)$ for every $\sigma \in \Sigma$.

Now let $\spe{S}$ be a geometric Hopf submonoid. Given a $\spe{H}$-structure $\spe{h}$, we let $\Gamma_{\spe{S}}(\spe{h})$ be the subcomplex of $\Sigma(\spe{h})$ consisting of chains $\emptyset \subset S_1 \subseteq \cdots \subseteq S_k \subset N$ such that $h|_{S_i} / S_{i-1}$ is not a $\spe{S}$-structure for some $i \in [k+1]$, where we define $S_{k+1} = N$. This generalizes the coloring complex of a graph as introduced by Steingr\'imsson \cite{steingrimsson}. If we let $\spe{g}$ be the acyclic mixed graph on the left in Figure \ref{fig:graph}, and we let $\spe{S}$ be the geometric Hopf submonoid of edgeless graphs, then $\Gamma_{\spe{S}}(\spe{g})$ is depicted in Figure \ref{fig:graphsubcomplex}.

\begin{figure}
\begin{center}
\begin{tikzpicture}[z={(-.5,-.28)}]


\node[thick,  draw=black, fill = white, circle] (bc) at (-3,0) {bc};
\node[thick,  draw=black, fill = white, circle] (b) at (-1.5,-2) {b};
\node[thick,  draw=black, fill = white, circle] (abd) at (1.5,-2) {abd};
\node[thick,  draw=black, fill = white, circle] (ad) at (3,0) {ad};
\node[thick,  draw=black, fill = white, circle] (a) at (1.5,2) {a};
\node[thick,  draw=black, fill = white, circle] (abc) at (-1.5,2) {abc};
\draw[-, line width = 1mm, line cap = round] (bc) -- (abc) -- (a) -- (ad) -- (abd) -- (b) -- (bc);
\draw[-, line width = 1mm, line cap = round] (b) -- (abc);
\draw[-, line width = 1mm, line cap = round] (a) -- (abd);

\end{tikzpicture}
\end{center}
\caption{An example of a subcomplex.}
\label{fig:graphsubcomplex}
\end{figure}
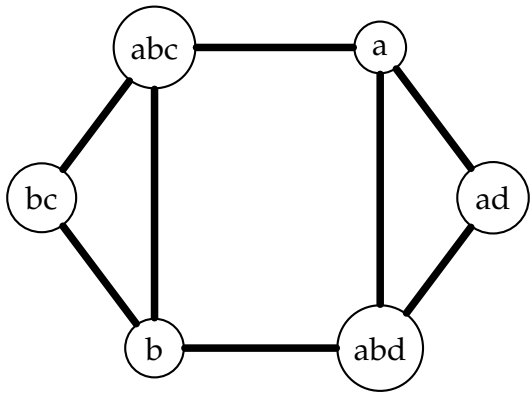

Our main reason for studying relative simplicial complexes is that the characteristic polynomial for $\spe{S}$ is essentially a Hilbert polynomial \cite{white-2}. This is also the reason why refer to $\spe{S}$ as a \emph{geometric} Hopf submonoid. That is, given an $\spe{H}$-structure $\spe{h}$, $\chi^{\spe{S}}_{\spe{H}}(\spe{h},k+1)$ is the Hilbert polynomial associated to the \emph{double cone} over the relative simplicial complex $(\Sigma(\spe{h}), \Gamma_{\spe{S}}(\spe{h}))$.

We say that $\spe{h}$ is \emph{relatively Cohen-Macaulay} (with respect to $\spe{S}$) if $(\Sigma(\spe{h}), \Gamma_{\spe{S}}(\spe{h}))$ is relatively Cohen-Macaulay. This implies that $\chi_{\spe{H}}^{\spe{S}}(\spe{h},k+1)$ is $h$-positive and $\Psi_{\spe{H}}^{\spe{S}}(\spe{h})$ is $F$-positive.

Now we apply Theorem \ref{thm:main} to the Hopf monoid of acyclic mixed graphs.
Let $\spe{S}$ be the Hopf submonoid of edgless graphs. Let $N$ be a finite set, and let $\spe{g}$ be an acyclic mixed graph. Then $\Gamma_{\spe{S}}(\spe{g})$ is always a connected pure simplicial complex of dimension $|N| - 1$. Since minors of acyclic mixed graphs are also acyclic mixed graphs, the condition of Theorem \ref{thm:main} holds.
\begin{theorem}
Let $\spe{S}$ be the geometric Hopf submonoid of edgeless graphs. Let $\spe{g}$ be an acyclic mixed graph. Then $\spe{g}$ is relatively Cohen-Macaulay with respect to $\spe{S}$. Moreover, $\bar{\chi}(\spe{g},k+1)$ is $h$-positive, and the strong chromatic symmetric function $\Psi_{\spe{M}}^{\spe{S}}(\spe{g})$ is $F$-positive.

\end{theorem}

\section{Another application to acyclic mixed graphs}

Another polynomial invariant associated to acyclic mixed graphs is the weak chromatic polynomial. The \emph{weak chromatic polynomial} $\chi(\spe{g}, k)$ counts the number of functions $f: V \to [k]$ subject to:
\begin{enumerate}
    \item if $(u,v)$ is a directed edge, then $f(u) \geq f(v)$,
    \item if $uv$ is an undirected edge, then $f(u) \neq f(v)$.
\end{enumerate}

The weak chromatic polynomial also comes from a geometric Hopf submonoid. Given a finite set $N$, let $\vec{\spe{D}}_N$ be the set of directed graphs. Then $\vec{\spe{D}}$ forms a geometric Hopf submonoid. However, it is not the case that every acyclic mixed graph is relatively Cohen-Macaulay with respect to $\vec{\spe{D}}$. The simplest example is given by the graph on the left in Figure \ref{fig:graph}. The reader can check that the $h$-vector for $\chi(\spe{g},k)$ is given by $(0,3,4,-1)$, which has negative entries.

Given an acyclic mixed graph $\spe{g}$, and two undirected edges $e$ and $f$, let $C$ be the smallest convex subset of $\spe{p}(\spe{g})$ containing the endpoints of $e$ and $f$, and let $I$ be the smallest ideal containing the endpoints of $e$. We call $e$ and $f$ \emph{crossing} if the $C \cap I$ contains one vertex of $f$ and not the other. We show in the extended version that if $\spe{g}$ has a pair of crossing edges, then $\spe{g}$ has a minor where $\Gamma_{\vec{\spe{D}}}(\spe{g}|_T/S)$ which is disconnected. 

If $\spe{g}$ has no pair of crossing edges then $\spe{g}$ is \emph{noncrossing}. When $\spe{g}$ is noncrossing, $\Gamma_{\vec{\spe{D}}}(\spe{g})$ is connected, pure and of the correct dimension. It is clear that being noncrossing is closed under minors, so the same fact is true of the $\Gamma$-complex for every minor. Hence, Theorem \ref{thm:main} applies.

\begin{theorem}
Let $\spe{S}$ be the geometric Hopf submonoid of directed graphs. Let $\spe{g}$ be an acyclic mixed graph. Then $\spe{g}$ is relatively Cohen-Macaulay with respect to $\spe{S}$ if and only if $\spe{g}$ is noncrossing. Moreover, $\chi(\spe{g},k+1)$ is $h$-positive, and the corresponding weak chromatic symmetric quasifunction is $F$-positive.

\end{theorem}

\section{Double Posets}
As our second application, we consider the Hopf monoid of double posets. This is related to the Hopf algebra of double posets introduced by Malvenuto and Reutenauer \cite{dblposet}. A \emph{double poset} $\spe{p}$ is on $N$ is a triple $(N, <_1, <_2)$, where $<_1$ and $<_2$ are both strict partial orders. Grinberg \cite{grinberg} also has studied them, and showed that many examples of quasisymmetric functions in the literature are $\spe{p}$-partition enumerators for some double poset $\spe{p}$. 

A $\spe{p}$-partition is a function $f: \spe{p} \to \mathbb{N}$ subject to:
\begin{enumerate}
    \item if $x <_1 y$, then $f(x) \leq f(y)$.
    \item if $x <_1 y$ and $y <_2 x$, then $f(x) < f(y)$.
\end{enumerate}
The $\spe{p}$-partition enumerator is given by \[K(\spe{p}) = \sum_{f} \prod_{v \in N} x_{f(v)}\] and the corresponding order polynomial is $\Omega(\spe{p}, k)$ which counts the number of $\spe{p}$-partitions $f$ with $f(x) \leq k$ for all $x$.

We show how the $\spe{p}$-partition enumerator does arise from the theory of linearized Hopf monoids.
Let $\spe{D}_N$ be the set of double posets on $N$. Given $\spe{p} \in \spe{D}_M$ and $\spe{q} \in \spe{D}_N$, where $M$ and $N$ are disjoint sets, we define a double poset $\spe{p} \cdot \spe{q} = (M \sqcup N, <_1, <_2)$. For $x, y \in M \sqcup N$, we say $x <_1 y$ if one of the following holds:
\begin{enumerate}
    \item $x, y \in M$ and $x <_1 y$ in $\spe{p}$.
    \item $x,y \in N$ and $x <_2 y$ in $\spe{q}$.
\end{enumerate}
For $x, y \in M \sqcup N$, we say $x <_2 y$ if one of the following holds:
\begin{enumerate}
    \item $x, y \in M$ and $x <_1 y$ in $\spe{p}$.
    \item $x,y \in N$ and $x <_2 y$ in $\spe{q}$.
        \item $x \in M$ and $y \in N$.

\end{enumerate}
Hence we have a multiplication operation for $\spe{D}$. Now we define the comultiplication.
Let $\spe{p} \in \spe{D}_{M \sqcup N}$. We define $\spe{p}|_M$ as follows: for $x, y \in M$, we say $x <_i y$ in $\spe{p}|_M$ if and only if $x <_i y$ in $\spe{p}$. We define $\Delta_{M,N}(\spe{p}) = \spe{p}|_M \otimes \spe{p}|_N$ if $M$ is an order ideal of $<_1$.

\begin{proposition}
The species $\spe{D}$ of double posets is a Cohen-Macaulay Hopf monoid.
\end{proposition}
It is not hard to see that $\spe{D}$ is a linearized Hopf monoid. For a double poset $\spe{p}$, we see that $\Sigma(\spe{p}) = \Sigma(<_1)$. In particular, $\spe{D}$ is Cohen-Macaulay.

 We let $\spe{S}_N$ consist of double posets $\spe{p}$ on $N$ for which there does not exist a pair $(x,y)$ with $x <_1 y$ and $y <_2 x$. We call such pairs \emph{inversions}. If $y$ covers $x$ with respect to $<_1$, then we call it a \emph{descent}. Then $\spe{S}$ forms a geometric Hopf submonoid. Then $\Psi_{\spe{H}}^{\spe{S}}(\spe{p}) = K(\spe{p})$.

\begin{figure}
\begin{center}
\begin{tikzpicture}
\node[thick, draw=black,  fill = white, circle] (a) at (-1,1.5) {a};
\node[thick, draw=black,  fill = white, circle] (b) at (-1,0) {b};
\node[thick, draw=black,  fill = white, circle] (c) at (-1,-1.5) {c};

\draw[->, thick, line cap = round] (a) -- (b);
\draw[->, thick, line cap = round] (b) -- (c);

\node[thick, draw=black,  fill = white, circle] (a2) at (2,-1) {a};
\node[thick, draw=black,  fill = white, circle] (c2) at (3,0) {b};
\node[thick,  draw=black, fill = white, circle] (d2) at (2,1) {c};


\draw[->, thick, line cap = round] (d2) -- (a2);
\end{tikzpicture}
\end{center}
\caption{A double poset. The left diagram is $<_1$, while the right is $<_2.$}
\label{fig:dblposet}
\end{figure}
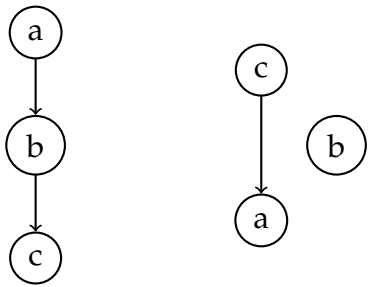

Of course, it is not the case that $K(\spe{p})$ is always $F$-positive, and hence $\spe{p}$ is not always Cohen-Macaulay with respect to $\spe{S}$. For example, $K(\spe{p}) = F_{2,1} + F_{1,2} - F_{1,1,1}$ for the double poset in Figure \ref{fig:dblposet}.
A double poset satisfies the \emph{inversion-to-descent condition} if whenever $(x,y)$ is and inversion, then there exists a descent $(w,z)$ with $x \leq_1 w$ and $z \leq_1 y$. It turns out that the inversion-to-descent condition is more general than the notion of tertispecial introduced by Grinberg.

We show in the full version that if $\spe{p}$ does not satisfy the inversion-to-descent condition, then it has a minor whose subcomplex $\Gamma(\spe{p}|_T/S)$ has dimension smaller than expected. It turns out that this is the only obstruction.

\begin{theorem}
Let $\spe{S}$ be the geometric Hopf submonoid of double posets $\spe{p}$ such that do not have inversions.
Let $\spe{p}$ be a double poset. Then $\spe{p}$ is relatively Cohen-Macaulay with respect to $\spe{S}$ if and only if $\spe{p}$ satisfies the inversion-to-descent condition. Moreover, $K(\spe{p})$ is $F$-positive and $\Omega(\spe{p}, k+1)$ is $h$-positive.

\end{theorem}

\section{Elements of the proof}
\label{sec:proof}
The hard part of the proof of Theorem \ref{thm:main} is the converse direction. Suppose that $\spe{H}$ is a Cohen-Macaulay Hopf monoid, and let $\spe{S}$ be a geometric Hopf submonoid. Let $\spe{h}$ satisfy the conditions of the converse direction. In the full version, we prove the converse in stages. First, we show that it suffices to prove that $\dim \widetilde{H}_i(\Gamma_{\spe{S}}(\spe{h}|_T/S)) = 0$ for $i < \dim \Sigma(\spe{h}|_T/S) - 1$. The full result follows by using Mayer-Vietoris exact sequences to reduce the study of links to joins of complexes arising from minors.

We also claim the following result:
if $\Gamma_{\spe{S}}(\spe{h})$ is connected, pure, and has dimension $|N| - 1$, then $\dim \widetilde{H}_i(\Gamma_{\spe{S}}(\spe{h})) = 0$ for $i < |N|-1$. This would imply the previous claim. We prove this stronger claim by focusing on the types of relative simplicial complexes that could possibly arise from our construction. 

Let $M$ be a collection of subsets of $[n]$, including $\emptyset$ and $[n]$, ordered by inclusion. We let $\Int(M)$ be the set of intervals, ordered by inclusion. Let $\mathcal{F}$ be an order filter in $\Int(M)$. Then define $\Sigma(M)$ to be the order complex of $M \setminus \{\emptyset, [n] \}$, and $\Gamma(\mathcal{F}, M)$ to consist of those chains $S_1 \subseteq S_2 \cdots \subseteq S_k$ such that $[S_i, S_{i+1}] \in \mathcal{F}$ for some $i \in [k+1]$, where we let $S_{k+1} = [n]$. 

\begin{lemma}
Suppose $\Sigma(M)$ is Cohen-Macaulay. Suppose the minimal elements of $\mathcal{F}$ have length $2$. If $\Gamma(\mathcal{F}, M)$ is connected, then $\dim \widetilde{H}_i(\Gamma(\mathcal{F}, M)) = 0$ for $i < \dim(\Sigma(M)) -1$. 
\end{lemma}

\begin{proof}[Proof Sketch]

The proof proceeds by induction on the number of intervals of length $2$ that are not in $\mathcal{F}$. In the base case, $\Gamma(\mathcal{F}, M)$ is the $(n-1)$-skeleton of $\Sigma(M)$, and hence is Cohen-Macaulay. 

Since $\mathcal{F}$ is non-empty, we have an interval of length two in $\mathcal{F}$, and an interval of length two that is not in $\mathcal{F}.$ The main idea is to find another pair $(\mathcal{F}', M')$ such that the intersection and union of $\Gamma(\mathcal{F}, M)$ and $\Gamma(\mathcal{F}', M')$ are also complexes of the form $\Gamma(\mathcal{F}'',M'')$ for some $\mathcal{F}''$ and $M''$. Then we use induction and a Mayer-Vietoris exact sequence for the pair $\Gamma(\mathcal{F}, M)$ and $\Gamma(\mathcal{F}', M')$ in order to compute the reduced homology of $\Gamma(\mathcal{F}, M)$. There are some technicalities involved to construct these things carefully, and to make sure everything is connected. \end{proof}

For an arbitrary Cohen-Macaulay Hopf monoid $\spe{H}$, and a Hopf submonoid $\spe{S}$, let $\spe{h}$ be an $\spe{H}$-structure on $[n]$. Then we let $M = \{S \subseteq [n]: \Delta_{S, [n] \setminus S}(\spe{h}) \neq 0 \}$. Then $\Sigma(M) = \Sigma(\spe{h})$. We let $\mathcal{F} = \{ [S,T]: \spe{h}|_T / S \not\in \spe{S}_{T \setminus S} \}$. Then $\Gamma(\mathcal{F}, M) = \Gamma_{\spe{S}}(\spe{h})$. If the latter is pure of codimension $1$, then the minimal elements of $\mathcal{F}$ consist only of intervals of length $2$, and the lemma applies. Hence we find that homology is concentrated in top dimension for $\Gamma_{\spe{S}}(\spe{h})$. So we see that if $\Gamma_{\spe{S}}(\spe{h})$ is connected, pure, and has codimension $1$, then its homology is concentrated in top dimension. 

\section{Future Work}
Obviously, there are many open questions:
\begin{enumerate}
    \item If $\Sigma(\spe{h})$ is shellable for all $\spe{h}$, and the conditions of Theorem \ref{thm:main} are met, then is $\Gamma_{\spe{S}}(\spe{h})$ shellable? Is the pair relatively shellable?
    \item Under what conditions does $\Gamma_{\spe{S}}(\spe{h})$ have a convex ear decomposition \cite{chari}?
    \item In all of our examples, $\Sigma(\spe{h})$ is contractible, but we actually know that $\Sigma(\spe{h})$ is shellable. Is it the case that $\Sigma(\spe{h})$ being contractible for all $\spe{h}$ implies that they are shellable?
\end{enumerate}
\input{hopfmon.bbl}

\end{document}

%% file: hopfmon.bbl
\newcommand{\etalchar}[1]{$^{#1}$}